\documentclass{amsart}
\usepackage{amssymb}
\usepackage{amsfonts}
\usepackage{amsmath}

\newtheorem{theorem}{Theorem}[section]
\newtheorem{proposition}{Proposition}[section]
\newtheorem{lemma}[theorem]{Lemma}
\theoremstyle{definition}

\theoremstyle{remark}

\theoremstyle{corollary}

\numberwithin{equation}{section}

\begin{document}
\title[Quasilinear singular elliptic systems with variable exponents]{%
Existence and a priori estimates of solutions for quasilinear singular
elliptic systems with variable exponents}

\begin{abstract}
This article sets forth results on the existence, a priori estimates and
boundedness of positive solutions of a singular quasilinear systems of
elliptic equations involving variable exponents. The approach is based on
Schauder's fixed point Theorem. A Moser iteration procedure is also obtained
for singular cooperative systems involving variable exponents establishing a
priori estimates and boundedness of solutions.
\end{abstract}

\author{Abdelkrim Moussaoui}
\address{Abdelkrim Moussaoui\\
Biology department, A. Mira Bejaia University, Targa Ouzemour 06000 Bejaia,
Algeria}
\email{abdelkrim.moussaoui@univ-bejaia.dz}
\author{Jean V\'{e}lin}
\address{Jean V\'{e}lin\\
D\'{e}partement de Math\'{e}matiques et Informatique, Laboratoire LAMIA,
Universit\'{e} des Antilles, Campus de Fouillole 97159 Pointe-\`{a}- Pitre,
Guadeloupe (FWI)}
\email{jean.velin@univ-antilles.fr}
\maketitle

\section{Introduction}

In the present paper we focus on the system of quasilinear elliptic equations%
\begin{equation}
\left\{ 
\begin{array}{ll}
-\Delta _{p(x)}u=f(u,v) & \text{in }\Omega \\ 
-\Delta _{q(x)}v=g(u,v) & \text{in }\Omega \\ 
u,v>0 & \text{in }\Omega \\ 
u,v=0 & \text{on }\partial \Omega ,%
\end{array}%
\right.  \tag{$P$}  \label{p}
\end{equation}%
on a bounded domain $\Omega $ in $%
%TCIMACRO{\U{211d} }%
%BeginExpansion
\mathbb{R}
%EndExpansion
^{N}$ $\left( N\geq 2\right) $ with Lipschitz boundary $\partial \Omega $,
which exhibits a singularity at zero. Here $\Delta _{p(x)}$ (resp. $\Delta
_{q(x)})$ stands for the $p(x)$-Laplacian (resp. $q(x)$-Laplacian)
differential operator on $W_{0}^{1,p(x)}(\Omega )$ (resp. $%
W_{0}^{1,q(x)}(\Omega )$) with $p,q:\Omega \rightarrow \lbrack 1,\infty ),$ 
\begin{equation}
\begin{array}{l}
1<p^{-}\leq p^{+}<N\text{ \ and \ }1<q^{-}\leq q^{+}<N,%
\end{array}
\label{33}
\end{equation}%
which satisfy the log-H\"{o}lder continuous condition, i.e., there is
constants $C_{1},C_{2}>0$ such that%
\begin{equation}
\begin{array}{l}
|p(x_{1})-p(y_{1})|\leq \frac{C_{1}}{-\ln |x_{1}-y_{1}|}\text{ \ and \ }%
|q(x_{2})-q(y_{2})|\leq \frac{C_{2}}{-\ln |x_{2}-y_{2}|},%
\end{array}
\label{16}
\end{equation}%
for every $x_{i},y_{i}\in \Omega $ with $|x_{i}-y_{i}|<1/2$, $i=1,2$.

Throught out this paper, we denote by $p^{\ast }$ and $q^{\ast }$ the
Sobolev critical exponents 
\begin{equation*}
\begin{array}{l}
p^{\ast }(x)=\frac{Np(x)}{N-p(x)}\text{ \ and \ }q^{\ast }(x)=\frac{Nq(x)}{%
N-q(x)}%
\end{array}%
\end{equation*}%
and we set%
\begin{equation*}
\begin{array}{l}
s^{-}=\inf_{x\in \Omega }s(x)\text{ \ and \ }s^{+}=\sup_{x\in \Omega }s(x).%
\end{array}%
\end{equation*}

A solution $(u,v)\in W_{0}^{1,p(x)}(\Omega )\times W_{0}^{1,q(x)}(\Omega )$
of problem (\ref{p}) is understood in the weak sense, that is, it satisfies 
\begin{equation}
\left\{ 
\begin{array}{l}
\int_{\Omega }|\nabla u|^{p(x)-2}\nabla u\nabla \varphi \,dx=\int_{\Omega
}f(u,v)\varphi \,dx \\ 
\int_{\Omega }|\nabla v|^{q(x)-2}\nabla v\nabla \psi \,dx=\int_{\Omega
}g(u,v)\psi \,dx\text{,}%
\end{array}%
\right.  \label{3}
\end{equation}%
for all $(\varphi ,\psi )\in W_{0}^{1,p(x)}(\Omega )\times
W_{0}^{1,q(x)}(\Omega )$.

\bigskip

Nonlinear boundary value problems involving $p(x)$-Laplacian operator are
mathematically challenging and important for applications. Their study is
stimulated by their applications in physical phenomena related to
electrorheological fluids and image restorations, see for instance \cite%
{Acerbi1,Acerbi2,Chen,Ru}. When $p(x)\equiv p$ and $q(x)\equiv q$ are
constant functions, $\Delta _{p(x)}$ and $\Delta _{q(x)}$ coincide with the
well-known $p$-Laplacian and $q$-Laplacian operators. However, it is worth
pointing out that $p(x)$-Laplacian operator possesses more complicated
nonlinearity than $p$-Laplacian since it is inhomogeneous and in general, it
has no first eigenvalue, that is, the infimum of the eigenvalues of $p(x)$%
-Laplacian equals $0$ (see, e.g., \cite{FZZ2, MV}). This point constitute a
serious technical difficulty in the study of problem (\ref{p}), for which
topological methods are difficult to apply. Another serious difficulty
encountered in studying system (\ref{p}) is that the nonlinearities $f(u,v)$
and $g(u,v)$ can exhibit singularities when the variables $u$ and $v$
approach zero. Specifically, we assume that $f,g:(0,+\infty )\times
(0,+\infty )\rightarrow (0,+\infty ),$ are continuous functions satisfying
the conditions:

\begin{description}
\item[\textrm{(H.}$f$\textrm{)}] 
\begin{equation*}
f(s_{1},s_{2})\leq m_{1}(1+s_{1}^{\alpha _{1}(x)})(1+s_{2}^{\beta _{1}(x)})\ 
\text{ for all }s_{1},s_{2}>0,
\end{equation*}%
with a constant $m_{1}>0$ and continuous functions $\alpha _{1},\beta _{1}:%
\overline{\Omega }\longrightarrow 
%TCIMACRO{\U{211d} }%
%BeginExpansion
\mathbb{R}
%EndExpansion
^{\ast }$.

\item[\textrm{(H.}$g$\textrm{)}] 
\begin{equation*}
g(s_{1},s_{2})\leq m_{2}(1+s_{1}^{\alpha _{2}(x)})(1+s_{2}^{\beta _{2}(x)})%
\text{ \ for all }s_{1},s_{2}>0,
\end{equation*}%
with a constant $m_{2}>0$ and continuous functions $\alpha _{2},\beta _{2}:%
\overline{\Omega }\longrightarrow 
%TCIMACRO{\U{211d} }%
%BeginExpansion
\mathbb{R}
%EndExpansion
^{\ast }$.
\end{description}

\bigskip

We explicitly observe that under assumptions \textrm{(H.}$f$\textrm{)} and%
\textrm{\ (H.}$g$) and depending on the sign of the variable exponents $%
\alpha _{i}(\cdot )$ and $\beta _{i}(\cdot ),$ $i=1,2$, system (\ref{p})
presents two types of complementary structures: 
\begin{equation}
\alpha _{2}^{-},\beta _{1}^{-}>0\text{ \ (cooperative structure),}
\label{h1}
\end{equation}%
\begin{equation}
\alpha _{2}^{+},\beta _{1}^{+}<0\text{ \ (competitive structure).}
\label{h2}
\end{equation}

If (\ref{h1}) holds, we assume

\begin{description}
\item[\textrm{H(}$f,g$\textrm{)}$_{1}$] 
\begin{equation*}
\begin{array}{c}
\sigma :=\min
\{\inf_{s_{1},s_{2}>0}f(s_{1},s_{2}),\inf_{s_{1},s_{2}>0}g(s_{1},s_{2})\}>0.%
\end{array}%
\end{equation*}
\end{description}

This assumption is useful in the subsequent estimates keeping the values of $%
f(s_{1},s_{2})$ and $g(s_{1},s_{2})$ above zero. In the case of competitive
system (\ref{p}), in addition of (\ref{h2}), we assume

\begin{description}
\item[\textrm{H(}$f,g$\textrm{)}$_{2}$] For all constant $M>0$ it hold%
\begin{equation*}
\begin{array}{l}
\lim_{s_{1}\rightarrow 0}\frac{f(s_{1},s_{2})}{s_{1}^{p^{-}-1}}=+\infty 
\text{ \ for all }s_{2}\in (0,M)%
\end{array}%
\end{equation*}%
and%
\begin{equation*}
\begin{array}{l}
\lim_{s_{2}\rightarrow 0}\frac{g(s_{1},s_{2})}{s_{2}^{q^{-}-1}}=+\infty 
\text{ \ for all }s_{1}\in (0,M).%
\end{array}%
\end{equation*}
\end{description}

\bigskip

This type of problem is rare in the literature. Actually, according to our
knowledge, the only class of singular problems incorporated in statement (%
\ref{p}) patterns the system for $f(u,v)=u^{\alpha _{1}(x)}v^{\beta _{1}(x)}$
and $g(u,v)=u^{\alpha _{2}(x)}v^{\beta _{2}(x)}$ was studied recently by
Alves \& Moussaoui \cite{AM}. The authors obtained the existence of
solutions through new theorems involving sub and supersolutions for singular
systems with variable exponents by dealing with cooperative and competitive
structures. However, when the exponent variable functions $p(\cdot ),q(\cdot
),\alpha _{i}(\cdot )$ and $\beta _{i}(\cdot )$, $i=1,2$, are reduced to be
constants, problem (\ref{p}) have been thoroughly investigated, we refer to 
\cite{MM}\ for system (\ref{p}) with cooperative structure, while we quote 
\cite{MM2, MM3} for the study of competitive structure in (\ref{p}).
Furthermore, in the constant exponent context, the singular problem (\ref{p}%
) arise in several physical situations such as fluid mechanics,
pseudoplastics flow, chemical heterogeneous catalysts, non- Newtonian
fluids, biological pattern formation, for more details about this subject,
we cite the papers of Fulks \& Maybe \cite{FM}, Callegari \& Nashman \cite%
{CN1,CN2} and the references therein.

\bigskip

Our goal is to establish the existence and regularity of (positive)
solutions for problem (\ref{p}) by processing the cases (\ref{h1}) and (\ref%
{h2}) related to the structure of (\ref{p}). Our main results are stated as
follows.

\begin{theorem}
\label{T1}Let assumptions \textrm{(H.}$f$\textrm{), (H.}$g$), \textrm{H(}$f,g
$\textrm{)}$_{1}$ and (\ref{h1}) hold with%
\begin{equation}
\begin{array}{c}
\beta _{1}(x)\leq \frac{q^{\ast }(x)}{p^{\ast }(x)}(p^{\ast }(x)-1),\text{ \ 
}\alpha _{2}(x)\leq \frac{p^{\ast }(x)}{q^{\ast }(x)}(q^{\ast }(x)-1)%
\end{array}
\label{c1*}
\end{equation}%
and%
\begin{equation}
\left\{ 
\begin{array}{l}
-\frac{1}{N}<\alpha _{1}^{-}\leq \alpha _{1}^{+}<0 \\ 
-\frac{1}{N}<\beta _{2}^{-}\leq \beta _{2}^{+}<0.%
\end{array}%
\right.   \label{c1}
\end{equation}%
Then, problem (\ref{p}) possesses at least one (positive) solution in $C^{1}(%
\overline{\Omega })\times C^{1}(\overline{\Omega })$ satisfying%
\begin{equation}
u(x),v(x)\geq c_{0}d(x),  \label{44}
\end{equation}%
where $d(x):=d(x,\partial \Omega )$ and $c_{0}$ is a positive constant.
\end{theorem}

\begin{theorem}
\label{T3}Under assumptions \textrm{(H.}$f$\textrm{), (H.}$g$), \textrm{H(}$%
f,g$\textrm{)}$_{2}$ and (\ref{h2}) with%
\begin{equation}
\max \{-\frac{1}{N},-\alpha _{1}^{-}\}<\beta _{1}^{-}\leq \beta
_{1}^{+}<0<\alpha _{1}^{-}\leq \alpha _{1}^{+}<p^{-}-1  \label{c2}
\end{equation}%
and%
\begin{equation}
\max \{-\frac{1}{N},-\beta _{2}^{-}\}<\alpha _{2}^{-}\leq \alpha
_{2}^{+}<0<\beta _{2}^{-}\leq \beta _{2}^{+}<q^{-}-1,  \label{c2*}
\end{equation}%
problem (\ref{p}) possesses at least one (positive) solution $(u,v)$ in $%
C^{1}(\overline{\Omega })\times C^{1}(\overline{\Omega })$ satisfying (\ref%
{44}).
\end{theorem}

\bigskip

The main technical difficulty consists in the presence of $p(x)$-Laplacian
and $q(x)$-Laplacian operators in the principle parts of equations in (\ref%
{p}) on the one hand and, on the other the presence\ of singular terms
through variable exponents that can occur under hypotheses \textrm{(H.}$f$%
\textrm{) }and \textrm{(H.}$g$). Under cooperative structure (\ref{h1}), by
adapting Moser iterations procedure to problem (\ref{p}), together with an
adequate truncation, we prove a priori estimates for an arbitrary solution
of (\ref{p}). In particular, it provides that all solution $(u,v)$ of (\ref%
{p}) are bounded in $L^{\infty }(\Omega )\times L^{\infty }(\Omega )$.
Taking advantage of this boundedness and applying Schauder's fixed point
Theorem we obtain the existence of a solution of problem (\ref{p}). To the
best of our knowledge, it is for the first time when Moser iterations method
is applied for problems with variable exponents.

For system (\ref{p}) subjected to competitive structure (\ref{h2}), we
develop some comparison arguments which provide a priori estimates on
solutions of (\ref{p}). In turn, these estimates enable us to obtain our
main result by applying the Schauder's fixed point theorem. It is worth
noting that besides our method is different from that used by Alves \&
Moussaoui \cite{AM}, our assumptions, precisely \textrm{H(}$f,g$\textrm{)}$%
_{1}$, (\ref{c2}) and (\ref{c2*}), are not satisfied by hypotheses
considered there.

We indicate simple examples showing the applicability of Theorems \ref{T1}
and \ref{T3}. Related to system (\ref{p}) under assumptions above, we can
handle singular cooperative systems of the form%
\begin{equation*}
\left\{ 
\begin{array}{l}
-\Delta _{p(x)}u=(u^{\alpha _{1}(x)}+1)(v^{\beta _{1}(x)}+1)\text{ in }\Omega
\\ 
-\Delta _{q(x)}v=(u^{\alpha _{2}(x)}+1)(v^{\beta _{2}(x)}+1)\text{ in }\Omega
\\ 
u,v>0\text{ \ \ \ in }\Omega \\ 
u,v=0\text{ \ \ \ on }\partial \Omega ,%
\end{array}%
\right.
\end{equation*}%
and singular competitive systems of type%
\begin{equation*}
\left\{ 
\begin{array}{l}
-\Delta _{p(x)}u=v^{\alpha _{1}(x)}+v^{\beta _{1}(x)}\text{ in }\Omega \\ 
-\Delta _{q(x)}v=u^{\alpha _{2}(x)}+u^{\beta _{2}(x)}\text{ in }\Omega \\ 
u,v>0\text{ \ \ \ in }\Omega \\ 
u,v=0\text{ \ \ \ on }\partial \Omega ,%
\end{array}%
\right.
\end{equation*}%
with variable exponents $\alpha _{1},\alpha _{2},\beta _{1},\beta _{1}$ as
in hypotheses (\ref{c1*}), (\ref{c1}) and (\ref{c2}), (\ref{c2*}),
respectively.

\bigskip

The rest of this article is organized as follows. Section \ref{S2}\ deals
with a priori estimates and regularity of solutions of cooperative system (%
\ref{p}), whereas Section \ref{S3} presents comparison properties of
competitive system (\ref{p}). Sections \ref{S4} and \ref{S5} contain the
proof of Theorems \ref{T1} and \ref{T3}.

\section{A priori estimates and regularity}

\label{S2}

Let $L^{p(x)}(\Omega )$ be the generalized Lebesgue space that consists of
all measurable real-valued functions $u$ satisfying%
\begin{equation*}
\begin{array}{l}
\rho _{p(x)}(u)=\int_{\Omega }|u(x)|^{p(x)}dx<+\infty ,%
\end{array}%
\end{equation*}%
endowed with the Luxemburg norm%
\begin{equation*}
\begin{array}{l}
\left\Vert u\right\Vert _{p(x)}=\inf \{\tau >0:\rho _{p(x)}(\frac{u}{\tau }%
)\leq 1\}.%
\end{array}%
\end{equation*}%
The variable exponent Sobolev space $W_{0}^{1,p(\cdot )}(\Omega )$ is
defined by%
\begin{equation*}
\begin{array}{l}
W_{0}^{1,p(x)}(\Omega )=\{u\in L^{p(x)}(\Omega ):|\nabla u|\in
L^{p(x)}(\Omega )\}.%
\end{array}%
\end{equation*}%
The norm $\left\Vert u\right\Vert _{1,p(x)}=\left\Vert \nabla u\right\Vert
_{p(x)}$ makes $W_{0}^{1,p(x)}(\Omega )$ a Banach space. On the basis of (%
\ref{16}), the following embedding%
\begin{equation}
\begin{array}{l}
W_{0}^{1,p(x)}(\Omega )\hookrightarrow L^{r(x)}(\Omega )%
\end{array}
\label{20}
\end{equation}%
is continuous with $1<r(x)\leq p^{\ast }(x)$ (see \cite[Corollary 5.3]{D}).

\begin{lemma}
\label{L1}$(i)$ For any $u\in L^{p(x)}(\Omega )$ we have%
\begin{equation*}
\begin{array}{l}
\left\Vert u\right\Vert _{p(x)}^{p^{-}}\leq \rho _{p(x)}(u)\leq \left\Vert
u\right\Vert _{p(x)}^{p^{+}}\text{ \ if \ }\left\Vert u\right\Vert _{p(x)}>1,%
\end{array}%
\end{equation*}%
\begin{equation*}
\begin{array}{l}
\left\Vert u\right\Vert _{p(x)}^{p^{+}}\leq \rho _{p(x)}(u)\leq \left\Vert
u\right\Vert _{p(x)}^{p^{-}}\text{ \ if \ }\left\Vert u\right\Vert
_{p(x)}\leq 1.%
\end{array}%
\end{equation*}

$(ii)$ For $u\in L^{p(x)}(\Omega )\backslash \{0\}$ we have 
\begin{equation}
\left\Vert u\right\Vert _{p(x)}=a\text{ \ if and only if }\rho _{p(x)}(\frac{%
u}{a})=1.  \label{normro}
\end{equation}
\end{lemma}

The next result provides a priori estimates for an arbitrary solution of (%
\ref{p}) subjected to cooperative structure.

\begin{theorem}
\label{T2} Assume that (\ref{h1}) and the growth conditions \textrm{(H.}$f$%
\textrm{)} and\textrm{\ (H.}$g$\textrm{)} hold with 
\begin{equation}
\left\{ 
\begin{array}{l}
\alpha _{1}^{+}<0<\beta _{1}(x)\leq \frac{q^{\ast }(x)}{p^{\ast }(x)}%
(p^{\ast }(x)-1)\text{ } \\ 
\text{\ }\beta _{2}^{+}<0<\alpha _{2}(x)\leq \frac{p^{\ast }(x)}{q^{\ast }(x)%
}(q^{\ast }(x)-1)%
\end{array}%
\right. \text{ \ in }\Omega \text{.}  \label{h3}
\end{equation}%
Then there exist positive constants $C=C(m_{1},\beta _{1},N,\Omega ,p,q)$
and $C^{\prime }=C^{\prime }(m_{2},\alpha _{2},N,\Omega ,p,q)$ such that
every solution $(u,v)\in W_{0}^{1,p(x)}(\Omega )\times W_{0}^{1,q(x)}(\Omega
)$ of (\ref{p}) satisfies the estimate 
\begin{equation}
\begin{array}{c}
\left\Vert u\right\Vert _{\infty }\leq C\max (1,\Vert u\Vert _{p^{\ast
}(x)})^{p^{+}/p^{-}}(1+\max (1,\Vert v\Vert _{q^{\ast }(x)}^{\beta
_{1}^{+}}))^{\frac{1}{(p^{-})^{\ast }-p^{-}}},%
\end{array}
\label{E1}
\end{equation}%
\begin{equation}
\begin{array}{c}
\left\Vert v\right\Vert _{\infty }\leq C^{\prime }\max (1,\Vert v\Vert
_{q^{\ast }(x)})^{q^{+}/q^{-}}(1+\max (1,\Vert u\Vert _{p^{\ast
}(x)}^{\alpha _{2}^{+}}))^{\frac{1}{(q^{-})^{\ast }-q^{-}}}.%
\end{array}
\label{E2}
\end{equation}%
In particular, problem (\ref{p}) has only bounded solutions.
\end{theorem}

\begin{proof}
Let $\phi :%
%TCIMACRO{\U{211d} }%
%BeginExpansion
\mathbb{R}
%EndExpansion
\longrightarrow \lbrack 0,1]$ be a $C^{1}$ cut-off function such that 
\begin{equation*}
\phi (s)=\left\{ 
\begin{array}{l}
0\text{ if }s\leq 0, \\ 
1\text{ if }s\geq 1%
\end{array}%
\right. \text{ \ and }\phi ^{\prime }(s)\geq 0\text{ in }[0,1].
\end{equation*}%
Given $\delta >0,$ we define $\phi _{\delta }(t)=\phi (\frac{t-1}{\delta })$
for all $t\in 
%TCIMACRO{\U{211d} }%
%BeginExpansion
\mathbb{R}
%EndExpansion
$. It follows that \ 
\begin{equation}
\begin{array}{l}
\phi _{\delta }\circ z\in W_{0}^{1,p(x)}(\Omega )\text{ \ and \ }\nabla
(\phi _{\delta }\circ z)=(\phi _{\delta }^{\prime }\circ z)\nabla z,\text{ \
for }z\in W_{0}^{1,p(x)}(\Omega )\text{.}%
\end{array}
\label{5}
\end{equation}

Let $(u,v)\in W_{0}^{1,p(x)}(\Omega )\times W_{0}^{1,q(x)}(\Omega )$ be a
weak solution of (\ref{p}). Acting in the first equation in (\ref{3}) with
the test function $\varphi =(\phi _{\delta }\circ u)\varphi $ with $\varphi
\in W_{0}^{1,p(x)}(\Omega )$ and $\varphi \geq 0$ in $\Omega $, we obtain%
\begin{equation*}
\begin{array}{l}
\int_{\Omega }|\nabla u|^{p(x)-2}\nabla u\nabla ((\phi _{\delta }\circ
u)\varphi )\,dx=\int_{\Omega }f(u,v)(\phi _{\delta }\circ u)\varphi \,dx.%
\end{array}%
\end{equation*}%
Hence, by (\ref{5}), we get%
\begin{equation*}
\begin{array}{l}
\int_{\Omega }|\nabla u|^{p(x)}(\phi _{\delta }^{\prime }\circ u)\varphi
\,dx+\int_{\Omega }|\nabla u|^{p(x)-2}\nabla u\nabla \varphi \text{ }(\phi
_{\delta }\circ u)\,dx=\int_{\Omega }f(u,v)(\phi _{\delta }\circ u)\varphi
\,dx.%
\end{array}%
\end{equation*}%
Since $\phi _{\delta }^{\prime }\circ u\geq 0,$ it follows that%
\begin{equation*}
\begin{array}{l}
\int_{\Omega }|\nabla u|^{p(x)-2}\nabla u\nabla \varphi \text{ }(\phi
_{\delta }\circ u)\,dx\leq \int_{\Omega }f(u,v)(\phi _{\delta }\circ
u)\varphi \,dx.%
\end{array}%
\end{equation*}%
Letting $\delta \rightarrow 0$ we achieve 
\begin{equation}
\begin{array}{l}
\int_{\{{u>1}\}}\left\vert \nabla {u}\right\vert ^{p(x)-2}\nabla {u}\nabla {%
\varphi }\text{ }dx\leq \int_{\{{u>1}\}}f({u},v)\varphi \ dx,%
\end{array}
\label{4}
\end{equation}%
for all $\varphi \in W_{0}^{1,p(x)}(\Omega )$ with $\varphi \geq 0$ in $%
\Omega $. Repeating the same argument with the second equation in (\ref{p}),
we get%
\begin{equation}
\begin{array}{l}
\int_{\{{v>1}\}}\left\vert \nabla {v}\right\vert ^{q(x)-2}\nabla {v}\nabla {%
\psi }\text{ }dx\leq \int_{\{{v>1}\}}g(u,v)\varphi \ dx,%
\end{array}
\label{4*}
\end{equation}%
for all $\psi \in W_{0}^{1,q(x)}(\Omega )$ with $\psi \geq 0$ in $\Omega $.

Given $M>0$, define 
\begin{equation*}
\begin{array}{c}
u_{M}\left( x\right) =\min \left\{ u\left( x\right) ,M\right\} ,\text{ \ }%
v_{M}\left( x\right) =\min \left\{ v\left( x\right) ,M\right\} .%
\end{array}%
\end{equation*}%
Observe that $h(s)=s^{k_{1}^{-}p^{+}+1}$ is a $C^{1}$ function, $h(0)=0$ and
there is a constant $L>0$ such that $|h^{\prime }(s)|\leq L$ for all $0\leq
s\leq M$. By proceeding analogously to the proof of \cite[Proposition XI.5,
page 155]{B}, it follows that $u_{M}^{k_{1}^{-}p^{+}+1}\in
W_{0}^{1,p(x)}(\Omega )\cap L^{\infty }(\Omega )$. Similarly we get $v_{M}^{%
\bar{k}_{1}^{-}q^{+}+1}\in W_{0}^{1,q(x)}(\Omega )\cap L^{\infty }(\Omega )$.

Inserting $(\varphi ,\psi )=(u_{M}^{k_{1}^{-}p^{+}+1},v_{M}^{\bar{k}%
_{1}^{-}q^{+}+1})$ in (\ref{4}) and (\ref{4*}), where 
\begin{equation}
\left\{ 
\begin{array}{c}
\left( k_{1}(x)+1\right) p(x)=p^{\ast }(x) \\ 
\left( \bar{k}_{1}(x)+1\right) q(x)=q^{\ast }(x),%
\end{array}%
\right.  \label{40}
\end{equation}%
one has%
\begin{equation}
\begin{array}{l}
\int_{\{{u>1}\}}\left\vert \nabla {u}\right\vert ^{p(x)-2}\nabla {u}\nabla
(u_{M}^{k_{1}^{-}p^{+}+1})\text{ }dx\leq \int_{\{{u>1}\}}f({u}%
,v)u_{M}^{k_{1}^{-}p^{+}+1}\ dx%
\end{array}
\label{8}
\end{equation}%
and%
\begin{equation}
\begin{array}{l}
\int_{\{{v>1}\}}\left\vert \nabla {v}\right\vert ^{q(x)-2}\nabla {v}\nabla
(v_{M}^{\bar{k}_{1}^{-}q^{+}+1})\text{ }dx\leq \int_{\{{v>1}\}}g(u,v)v_{M}^{%
\bar{k}_{1}^{-}q^{+}+1}\ dx,%
\end{array}
\label{8*}
\end{equation}%
\bigskip

\noindent \textbf{Step 1. Estimation of the left-hand side in (\ref{8}) and (%
\ref{8*})\bigskip }

In what follows denote by $(s-1)^{+}:=\max \{s,1\}$ for $s\geq 0$.

First, observe that{%
\begin{equation}
\begin{array}{c}
\left\vert \nabla {u}_{M}\right\vert ^{p(x)}u_{M}{}^{k_{1}^{-}p(x)}=\frac{1}{%
(k_{1}^{-}+1)^{p(x)}}|\nabla (u_{M})^{k_{1}^{-}+1})|^{p(x)}\geq \frac{1}{%
(k_{1}^{-}+1)^{p^{+}}}|\nabla (u_{M})^{k_{1}^{-}+1})|^{p(x)}.%
\end{array}%
\end{equation}%
Then}%
\begin{equation}
\begin{array}{l}
\int_{\{{u>1}\}}\left\vert \nabla {u}\right\vert ^{p(x)-2}\nabla {u}\nabla
(u_{M}^{k_{1}^{-}p^{+}+1})\text{ }dx=\int_{\{{u>1}\}}\left\vert \nabla {u}%
\right\vert ^{p(x)-2}\nabla {u}\nabla (u_{M})^{k_{1}^{-}p^{+}+1}\text{ }dx
\\ 
=(k_{1}^{-}p^{+}+1)\int_{\{{u}_{M}{>1}\}}\left\vert \nabla {u}%
_{M}\right\vert ^{p(x)}u_{M}^{k_{1}^{-}p^{+}}\text{ }dx\geq
(k_{1}^{-}p^{+}+1)\int_{\{{u}_{M}{>1}\}}\left\vert \nabla {u}_{M}\right\vert
^{p(x)}u_{M}^{k_{1}^{-}p(x)}\text{ }dx \\ 
\geq \frac{k_{1}^{-}p^{+}+1}{(k_{1}^{-}+1)^{p^{+}}}\int_{\{{u}_{M}{>1}%
\}}|\nabla (u_{M}^{k_{1}^{-}+1})|^{p(x)}\text{ }dx.%
\end{array}
\label{1}
\end{equation}%
On the other hand{, using (\ref{normro}) and }through the mean value
theorem, there exists $x_{0}\in \Omega $ such that%
\begin{equation}
\begin{array}{l}
1=\int_{\Omega }\left\vert \frac{(u_{M}-1)^{+}}{\Vert (u_{M}-1)^{+}\Vert
_{(k_{1}^{-}+1)p^{\ast }(x)}}\right\vert ^{(k_{1}^{-}+1)p^{\ast }(x)}dx \\ 
=\int_{\Omega }\left\vert \frac{((u_{M}-1)^{+})^{k_{1}^{-}+1}}{\Vert
((u_{M}-1)^{+})^{k_{1}^{-}+1}\Vert _{p^{\ast }(x)}}\right\vert ^{p^{\ast
}(x)}\times \left( \frac{\Vert ((u_{M}-1)^{+})^{k_{1}^{-}+1}\Vert _{p^{\ast
}(x)}}{\Vert (u_{M}-1)^{+}\Vert _{(k_{1}^{-}+1)p^{\ast }(x)}^{k_{1}^{-}+1}}%
\right) ^{p^{\ast }(x)}dx \\ 
=\left( \frac{\Vert ((u_{M}-1)^{+})^{k_{1}^{-}+1}\Vert _{p^{\ast }(x)}}{%
\Vert (u_{M}-1)^{+}\Vert _{(k_{1}^{-}+1)p^{\ast }(x)}^{k_{1}^{-}+1}}\right)
^{p^{\ast }(x_{0})},%
\end{array}%
\end{equation}%
which implies%
\begin{equation}
\Vert ((u_{M}-1)^{+})^{k_{1}^{-}+1}\Vert _{p^{\ast }(x)}=\Vert
(u_{M}-1)^{+}\Vert _{(k_{1}^{-}+1)p^{\ast }(x)}^{k_{1}^{-}+1}.  \label{9}
\end{equation}%
{Furthermore, from (\ref{normro}) one has}%
\begin{equation*}
\begin{array}{l}
\int_{\Omega }|\frac{\nabla ((u_{M}-1)^{+})^{k_{1}^{-}+1}}{\Vert
((u_{M}-1)^{+})^{k_{1}^{-}+1}\Vert _{1,p(x)}}|^{p(x)}\text{ }dx=1.%
\end{array}%
\end{equation*}%
Using the mean value theorem, there exists $x_{M}\in \Omega $ such that%
\begin{equation}
\begin{array}{l}
\int_{\Omega }|\nabla ((u_{M}-1)^{+})^{k_{1}^{-}+1}|^{p(x)}\text{ }dx=\Vert
((u_{M}-1)^{+})^{k_{1}^{-}+1}\Vert _{1,p(x)}^{p(x_{M})}.%
\end{array}
\label{6}
\end{equation}%
Then, (\ref{2}), (\ref{9}), (\ref{6}) and through the Sobolev embedding (\ref%
{20}), one gets%
\begin{equation}
\begin{array}{l}
\frac{k_{1}^{-}p^{+}+1}{(k_{1}^{-}+1)^{p^{+}}}\int_{\{{u}_{M}{>1}\}}|\nabla
(u_{M}^{k_{1}^{-}+1})|^{p(x)}\text{ }dx=\frac{k_{1}^{-}p^{+}+1}{%
(k_{1}^{-}+1)^{p^{+}}}\int_{\Omega }|\nabla
((u_{M}-1)^{+})^{k_{1}^{-}+1}|^{p(x)}\text{ }dx \\ 
=\frac{k_{1}^{-}p^{+}+1}{(k_{1}^{-}+1)^{p^{+}}}\Vert
((u_{M}-1)^{+})^{k_{1}^{-}+1}\Vert _{1,p(x)}^{p(x_{M})}\geq \hat{C}_{1}\frac{%
k_{1}^{-}p^{+}+1}{(k_{1}^{-}+1)^{p^{+}}}\Vert
((u_{M}-1)^{+})^{k_{1}^{-}+1}\Vert _{p^{\ast }(x)}^{p(x_{M})} \\ 
=\hat{C}_{1}\frac{k_{1}^{-}p^{+}+1}{(k_{1}^{-}+1)^{p^{+}}}\Vert
(u_{M}-1)^{+}\Vert _{(k_{1}^{-}+1)p^{\ast }(x)}^{(k_{1}^{-}+1)p(x_{M})}\geq
C_{1}\frac{k_{1}^{-}p^{+}+1}{(k_{1}^{-}+1)^{p^{+}}}\Vert (u_{M}-1)^{+}\Vert
_{(k_{1}^{-}+1)p^{\ast }(x)}^{(k_{1}^{-}+1)p^{\pm }},%
\end{array}
\label{12}
\end{equation}%
where $C_{1}=C_{1}(p,N,\Omega )$ is a positive constant and 
\begin{equation}
p^{\pm }=\left\{ 
\begin{array}{ll}
p^{+} & \text{if }\Vert (u_{M}-1)^{+}\Vert _{(k_{1}^{-}+1)p^{\ast }(x)}>1 \\ 
p^{-} & \text{if }\Vert (u_{M}-1)^{+}\Vert _{(k_{1}^{-}+1)p^{\ast }(x)}\leq
1.%
\end{array}%
\right.
\end{equation}%
Similarly, following the same argument as above leads to 
\begin{equation}
\begin{array}{l}
\frac{\bar{k}_{1}^{-}q^{+}+1}{(\bar{k}_{1}^{-}+1)^{q^{+}}}\int_{\{{v}_{M}{>1}%
\}}|\nabla (v_{M}^{\bar{k}_{1}^{-}+1})|^{q(x)}\text{ }dx\geq C_{2}\frac{\bar{%
k}_{1}^{-}q^{+}+1}{(\bar{k}_{1}^{-}+1)^{q^{+}}}\Vert (v_{M}-1)^{+}\Vert _{(%
\bar{k}_{1}^{-}+1)q^{\ast }(x)}^{(\bar{k}_{1}^{-}+1)q^{\pm }},%
\end{array}
\label{12*}
\end{equation}%
with positive constants $C_{2}=C_{2}(q,N,\Omega )$ and%
\begin{equation}
q^{\pm }=\left\{ 
\begin{array}{ll}
q^{+} & \text{if }\Vert (v_{M}-1)^{+}\Vert _{(\bar{k}_{1}^{-}+1)q^{\ast
}(x)}>1 \\ 
q^{-} & \text{if }\Vert (v_{M}-1)^{+}\Vert _{(\bar{k}_{1}^{-}+1)q^{\ast
}(x)}\leq 1.%
\end{array}%
\right.
\end{equation}%
\bigskip

\noindent \textbf{Step 2. Estimation of the right-hand side in (\ref{8}) and
(\ref{8*}).\bigskip }

Using (\ref{4}), \textrm{(H.}$f$\textrm{)}, (\ref{40}), (\ref{h3}), (\ref{20}%
) together with H\"{o}lder's inequality and \cite[Proposition 2.3]{Benou},
we get 
\begin{equation}
\begin{array}{l}
\int_{\{{u>1}\}}f(u,v)u_{M}^{k_{1}^{-}p^{+}+1}\ dx\leq \int_{\{{u>1}%
\}}f(u,v)u^{k_{1}^{-}p^{+}+1}\ dx \\ 
\leq 2m_{1}\int_{\{{u>1}\}}(1+v^{\beta _{1}(x)})u^{k_{1}^{-}p^{+}+1}dx \\ 
=2m_{1}\int_{\Omega }((u-1)^{+})^{k_{1}^{-}p^{+}+1}\text{ }%
dx+2m_{1}\int_{\Omega }v^{\beta _{1}(x)}((u-1)^{+})^{k_{1}^{-}p^{+}+1}\text{ 
}dx \\ 
\leq \hat{C}_{2}\left( \left\Vert (u-1)^{+}\right\Vert _{p^{\ast
}(x)}^{k_{1}^{-}p^{+}+1}+\left\Vert (u-1)^{+}\right\Vert _{p^{\ast
}(x)}^{k_{1}^{-}p^{+}+1}\left\Vert v^{\beta _{1}(x)}\right\Vert _{\frac{%
p^{\ast }(x)}{p^{\ast }(x)-1}}\right) \\ 
\leq \hat{C}_{2}^{\prime }\left( \left\Vert (u-1)^{+}\right\Vert _{p^{\ast
}(x)}^{k_{1}^{-}p^{+}+1}+\left\Vert (u-1)^{+}\right\Vert _{p^{\ast
}(x)}^{k_{1}^{-}p^{+}+1}\left\Vert v\right\Vert _{\frac{\beta _{1}(x)p^{\ast
}(x)}{p^{\ast }(x)-1}}^{\beta _{1}^{\pm }}\right) \\ 
\leq C_{2}\left\Vert (u-1)^{+}\right\Vert _{p^{\ast
}(x)}^{k_{1}^{-}p^{+}+1}(1+\left\Vert v\right\Vert _{q^{\ast }(x)}^{\beta
_{1}^{\pm }}),%
\end{array}
\label{13}
\end{equation}%
with a positive constant $C_{2}=C_{2}(m_{1},\beta _{1},N,\Omega ,p,q)$ and 
\begin{equation}
\beta _{1}^{\pm }=\left\{ 
\begin{array}{ll}
\beta _{1}^{+} & \text{if }\left\Vert v\right\Vert _{q^{\ast }(x)}>1 \\ 
\beta _{1}^{-} & \text{if }\left\Vert v\right\Vert _{q^{\ast }(x)}\leq 1.%
\end{array}%
\right.  \label{14}
\end{equation}%
Similarly, by (\ref{4*}), \textrm{(H.}$g$\textrm{)}, (\ref{40}), (\ref{h3}),
(\ref{20}), combined with H\"{o}lder's inequality and \cite[Proposition 2.3]%
{Benou}, one has 
\begin{equation}
\begin{array}{c}
\int_{\{{v>1}\}}g(u,v)v_{M}^{\bar{k}_{1}^{-}q^{+}+1}\ dx\leq C_{2}^{\prime
}(1+\left\Vert u\right\Vert _{p^{\ast }(x)}^{\alpha _{2}^{i}})\left\Vert
v\right\Vert _{q^{\ast }(x)}^{\bar{k}_{1}^{-}q^{+}+1},%
\end{array}
\label{13*}
\end{equation}%
where the positive constant $C_{2}^{\prime }=C_{2}^{\prime }(m_{2},\alpha
_{2},N,\Omega ,p,q)$ and%
\begin{equation}
\alpha _{2}^{i}=\left\{ 
\begin{array}{ll}
\alpha _{2}^{+} & \text{if }\left\Vert u\right\Vert _{p^{\ast }(x)}>1 \\ 
\alpha _{2}^{-} & \text{if }\left\Vert u\right\Vert _{p^{\ast }(x)}\leq 1.%
\end{array}%
\right.  \label{14*}
\end{equation}%
\bigskip

\noindent \textbf{Step 3. Moser iteration procedure and passage to the
limit.\bigskip }

We note that if $\Vert (u-1)^{+}\Vert _{p^{\ast }(x)},\left\Vert
(v-1)^{+}\right\Vert _{q^{\ast }(x)}>1$, then there hold 
\begin{equation}
\begin{array}{c}
\left\Vert (u-1)^{+}\right\Vert _{p^{\ast }(x)}^{k_{1}^{-}p^{+}+1}\leq \Vert
(u-1)^{+}\Vert _{p^{\ast }(x)}^{(k_{1}^{-}+1)p^{+}}\ \ \text{and}\ \
\left\Vert (v-1)^{+}\right\Vert _{q^{\ast }(x)}^{\bar{k}_{1}^{-}q^{+}+1}\leq
\left\Vert (v-1)^{+}\right\Vert _{q^{\ast }(x)}^{(\bar{k}_{1}^{-}+1)q^{+}}%
\end{array}
\label{15}
\end{equation}%
because $p^{+},q^{+}>1$. Then, it follows from (\ref{12}) - (\ref{15}) that%
\begin{equation}
\begin{array}{l}
\Vert (u_{M}-1)^{+}\Vert _{(k_{1}^{-}+1)p^{\ast }(x)}\leq C^{\frac{1}{%
k_{1}^{-}+1}}\left( \frac{k_{1}^{-}+1}{(k_{1}^{-}p^{+}+1)^{\frac{1}{p^{+}}}}%
\right) ^{\frac{p^{+}}{(k_{1}^{-}+1)p^{\pm }}}\left\Vert
(u-1)^{+}\right\Vert _{p^{\ast }(x)}^{p^{+}/p^{-}}\left( 1+\left\Vert
v\right\Vert _{q^{\ast }(x)}^{\beta _{1}^{\pm }}\right) ^{\frac{1}{%
(p^{-})^{\ast }}}%
\end{array}
\label{25}
\end{equation}%
and%
\begin{equation}
\begin{array}{l}
\Vert (v_{M}-1)^{+}\Vert _{(\bar{k}_{1}^{-}+1)q^{\ast }(x)}\leq C^{\frac{1}{%
\bar{k}_{1}^{-}+1}}\left( \frac{\bar{k}_{1}^{-}+1}{(\bar{k}_{1}^{-}q^{+}+1)^{%
\frac{1}{q^{+}}}}\right) ^{\frac{q^{+}}{(\bar{k}_{1}^{-}+1)q^{\pm }}%
}\left\Vert (v-1)^{+}\right\Vert _{q^{\ast }(x)}^{q^{+}/q^{-}}\left(
1+\left\Vert u\right\Vert _{p^{\ast }(x)}^{\alpha _{2}^{\pm }}\right) ^{%
\frac{1}{(q^{-})^{\ast }}}%
\end{array}
\label{25*}
\end{equation}%
with a constant $C=C(m_{1},\alpha _{2},\beta _{1},N,\Omega ,p,q)>0$.

Inductively, we construct the sequences $\{k_{n}\}_{n\geq 1}$ and $\{%
\overline{k}_{n}\}_{n\geq 1}$ by defining 
\begin{equation}
\left\{ 
\begin{array}{l}
k_{n}(x)+1=(k_{n-1}(x)+1)\frac{p^{\ast }(x)}{p(x)}=\left( \frac{p^{\ast }(x)%
}{p(x)}\right) ^{n}, \\ 
\overline{k}_{n}(x)+1=(\overline{k}_{n-1}(x)+1)\frac{q^{\ast }(x)}{q(x)}%
=\left( \frac{q^{\ast }(x)}{q(x)}\right) ^{n},%
\end{array}%
\right.  \label{10}
\end{equation}%
for all $n\geq 2$ starting with (\ref{40}). If we have for infinitely many $%
n $ that 
\begin{equation*}
\begin{array}{c}
\Vert (u-1)^{+}\Vert _{(k_{n}^{-}+1)p^{\ast }(x)}\leq 1\ \ \text{and }\ \
\Vert (v-1)^{+}\Vert _{(\overline{k}_{n}^{-}+1)q^{\ast }(x)}\leq 1,%
\end{array}%
\end{equation*}%
then letting $n\rightarrow \infty $ we get $\Vert (u\Vert _{\infty }\leq 1$
and $\Vert v\Vert _{\infty }\leq 1$, and we are done. If not, it suffices to
consider the case 
\begin{equation*}
\Vert (u-1)^{+}\Vert _{(k_{n}^{-}+1)p(x)}>1\ \ \text{and }\ \ \Vert
(v-1)^{+}\Vert _{(\overline{k}_{n}^{-}+1)q(x)}>1
\end{equation*}%
for all $n$ because otherwise the proof reduces to special case of Moser
iteration procedure for an elliptic equation. In this case, we argue as for
obtaining (\ref{25}) and (\ref{25*}). Namely, proceeding by induction
through (\ref{10}) and then letting $M\rightarrow \infty $ we arrive at 
\begin{equation}
\begin{array}{l}
\Vert (u-1)^{+}\Vert _{(k_{n}^{-}+1)p^{\ast }(x)} \\ 
\leq C_{1}^{\frac{1}{k_{n}^{-}+1}}\left( \frac{k_{n}^{-}+1}{%
(k_{n}^{-}p^{+}+1)^{\frac{1}{p^{+}}}}\right) ^{\frac{p^{+}}{%
(k_{n}^{-}+1)p^{\pm }}}\Vert (u-1)^{+}\Vert _{(k_{n-1}^{-}+1)p^{\ast
}(x)}^{p^{+}/p^{-}}(1+\Vert v\Vert _{q^{\ast }(x)}^{\beta _{1}^{\pm }})^{%
\frac{1}{(k_{n-1}^{-}+1)(p^{-})^{\ast }}}%
\end{array}
\label{30}
\end{equation}%
and%
\begin{equation}
\begin{array}{l}
\Vert (v-1)^{+}\Vert _{(\overline{k}_{n}^{-}+1)q^{\ast }(x)} \\ 
\leq C_{2}^{\frac{1}{\overline{k}_{n}^{-}+1}}\left( \frac{\overline{k}%
_{n}^{-}+1}{(\overline{k}_{n}^{-}q^{+}+1)^{\frac{1}{q^{+}}}}\right) ^{\frac{%
q^{+}}{(\overline{k}_{n}^{-}+1)q^{\pm }}}\Vert (v-1)^{+}\Vert _{(\overline{k}%
_{n}^{-}+1)q^{\ast }(x)}^{q^{+}/q^{-}}(1+\Vert u\Vert _{p^{\ast
}(x)}^{\alpha _{2}^{\pm }})^{\frac{1}{(\overline{k}_{n-1}^{-}+1)(q^{-})^{%
\ast }}},%
\end{array}
\label{31}
\end{equation}%
with positive constants $C_{1}=C_{1}(N,\Omega ,m_{1},p,\beta _{1})$ and $%
C_{2}=C_{2}(N,\Omega ,m_{2},q,\alpha _{2})$. It turns out from (\ref{30})
that 
\begin{equation*}
\begin{array}{l}
\Vert (u-1)^{+}\Vert _{(k_{n}^{-}+1)p^{\ast }(x)} \\ 
\\ 
\leq C_{1}^{\sum\limits_{i=1}^{n}\frac{1}{k_{i}^{-}+1}}\left(
\prod\limits_{i=1}^{n}\left[ \left( \frac{k_{i}^{-}+1}{\left(
k_{i}^{-}p^{+}+1\right) ^{\frac{1}{p^{+}}}}\right) ^{\frac{1}{\sqrt{%
k_{i}^{-}+1}}}\right] ^{\frac{1}{\sqrt{k_{i}^{-}+1}}}\right) ^{p^{+}/p^{\pm
}}\Vert (u-1)^{+}\Vert _{p^{\ast }(x)}^{p^{+}/p^{-}}(1+\Vert v\Vert
_{q^{\ast }(x)}^{\beta _{1}^{\pm }})^{\frac{1}{(p^{-})^{\ast }}\left(
1+\sum\limits_{i=1}^{n-1}\frac{1}{k_{i}^{-}+1}\right) }.%
\end{array}%
\end{equation*}%
Furthermore, since $\lim_{z\rightarrow \infty }\left( \frac{z+1}{\left(
zp^{+}+1\right) ^{\frac{1}{p^{+}}}}\right) ^{\frac{1}{\sqrt{z+1}}}=1$, there
is a positive constant $C_{0}$ for which one has%
\begin{equation}
\begin{array}{c}
\Vert (u-1)^{+}\Vert _{(k_{n}+1)p^{\ast }}\leq C_{1}^{\sum\limits_{i=1}^{n}%
\frac{1}{k_{i}^{-}+1}}C_{0}^{\frac{p^{+}}{p^{\pm }}\sum\limits_{i=1}^{n}%
\frac{1}{\sqrt{k_{i}^{-}+1}}}\Vert (u-1)^{+}\Vert _{p^{\ast
}(x)}^{p^{+}/p^{-}}(1+\Vert v\Vert _{q^{\ast }(x)}^{\beta _{1}^{\pm }})^{%
\frac{1}{(p^{\ast })^{-}}\left( 1+\sum\limits_{i=1}^{n-1}\frac{1}{k_{i}^{-}+1%
}\right) }.%
\end{array}
\label{11}
\end{equation}%
Similarly, we obtain 
\begin{equation}
\begin{array}{c}
\Vert (v-1)^{+}\Vert _{(\overline{k}_{n}^{-}+1)q^{\ast }(x)}\leq
C_{2}^{\sum\limits_{i=1}^{n}\frac{1}{\overline{k}_{i}^{-}+1}}C_{0}^{\frac{%
q^{+}}{q^{\pm }}\sum\limits_{i=1}^{n}\frac{1}{\sqrt{\overline{k}_{i}^{-}+1}}%
}\Vert (v-1)^{+}\Vert _{q^{\ast }(x)}^{q^{+}/q^{-}}(1+\Vert u\Vert _{p^{\ast
}(x)}^{\alpha _{2}^{\pm }})^{\frac{1}{(q^{-})^{\ast }}\left(
1+\sum\limits_{i=1}^{n-1}\frac{1}{\bar{k}_{i}^{-}+1}\right) }.%
\end{array}
\label{11*}
\end{equation}%
Moreover, (\ref{10}) guarantees the convergence of the series in (\ref{11})
and (\ref{11*}), for instance%
\begin{equation*}
\begin{array}{l}
1+\sum\limits_{i=1}^{n-1}\frac{1}{\overline{k}_{i}^{-}+1}=\sum%
\limits_{i=0}^{n-1}\left( \frac{p^{-}}{(p^{-})^{\ast }}\right)
^{i}\longrightarrow \frac{(p^{-})^{\ast }}{(p^{-})^{\ast }-p^{-}}.%
\end{array}%
\end{equation*}%
Letting $n\rightarrow \infty $ in (\ref{11}) and (\ref{11*}) we derive the
estimates (\ref{E1}) and (\ref{E2}). This completes the proof.
\end{proof}

Next result is consequence of Theorem \ref{T2}.

\begin{proposition}
\label{P3}Under the assumptions of Theorem \ref{T1}, every solutions $(u,v)$
of (\ref{p}) is bounded in $C^{1,\gamma }(\overline{\Omega })\times
C^{1,\gamma }(\overline{\Omega })$ and there is a constant $R>0$ such that%
\begin{equation*}
\left\Vert u\right\Vert _{C^{1,\gamma }(\overline{\Omega })},\left\Vert
v\right\Vert _{C^{1,\gamma }(\overline{\Omega })}<R.
\end{equation*}%
Moreover, it holds%
\begin{equation}
u(x),v(x)\geq c_{0}d(x),  \label{63}
\end{equation}%
with some constant $c_{0}>0$.
\end{proposition}

\begin{proof}
We first show (\ref{63}). Recalling the constant $\sigma >0$ in \textrm{H(}$%
f,g$\textrm{)}$_{1}$, let $z_{1}$ and $z_{2}$ the only positive solutions of%
\begin{equation}
\left\{ 
\begin{array}{l}
-\Delta _{p(x)}z_{1}=\sigma \text{ in }\Omega \\ 
z_{1}=0\text{ on }\partial \Omega%
\end{array}%
\right. \ \text{and}\ \left\{ 
\begin{array}{l}
-\Delta _{q(x)}z_{2}=\sigma \text{ in }\Omega \\ 
z_{2}=0\text{ on }\partial \Omega ,%
\end{array}%
\right.  \label{64}
\end{equation}%
which are known to satisfy 
\begin{equation}
z_{1}(x)\geq c_{2}d(x)\text{ \ and \ }z_{2}(x)\geq c_{2}^{\prime }d(x)\text{
\ in }\Omega ,  \label{65}
\end{equation}%
for certain positive constants $c_{2}$ and $c_{2}^{\prime }$ (see, e.g., 
\cite{AM}). Then, from (\ref{p}), (\ref{64}) and \textrm{H(}$f,g$\textrm{)}$%
_{1}$, it follows that%
\begin{equation*}
\left\{ 
\begin{array}{l}
-\Delta _{p(x)}u\geq -\Delta _{p(x)}z_{1}\text{ in }\Omega \\ 
u=z_{1}\text{ on }\partial \Omega%
\end{array}%
\right. \text{ \ and \ }\left\{ 
\begin{array}{l}
-\Delta _{q(x)}v\geq -\Delta _{q(x)}z_{2}\text{ in }\Omega \\ 
v=z_{2}\text{ on }\partial \Omega .%
\end{array}%
\right.
\end{equation*}%
Therefore, the weak comparison principle leads to (\ref{63}).

By virtue of\textrm{\ (H.}$f$\textrm{)}, \textrm{(H.}$g$\textrm{)}, (\ref{63}%
), (\ref{h1}), (\ref{h2}) and (\ref{c1}), on account of Theorem \ref{T2},
one has%
\begin{equation}
\begin{array}{c}
f(u,v)\leq C_{0}d(x)^{\alpha _{1}^{-}}\text{ \ and \ }f(u,v)\leq
C_{0}^{\prime }d(x)^{\beta _{2}^{-}}\text{ \ in }\Omega ,%
\end{array}%
\end{equation}%
for some positive constants $C_{0}$ and $C_{0}^{\prime }$. Then, the $%
C^{1,\alpha }$-boundedness of $u$ and $v$ follows from \cite[Lemma 2]{AM}.
The proof is completed.
\end{proof}

\section{Comparison properties}

\label{S3}

In this section, we assume that (\ref{c2}) and (\ref{c2*}) hold. For a fixed 
$\delta >0$ small, define $\overline{u}$ and $\overline{v}$ in $C^{1,\gamma
}(\overline{\Omega }),$ for certain $\gamma \in (0,1)$, as the unique weak
solutions of the problems%
\begin{equation}
-\Delta _{p(x)}\overline{u}=\lambda \left\{ 
\begin{array}{ll}
1 & \text{ in \ }\Omega \backslash \overline{\Omega }_{\delta } \\ 
\overline{u}^{-\alpha _{1}(x)} & \text{\ in \ }\Omega _{\delta }%
\end{array}%
\right. ,\text{ }\overline{u}>0\text{ in }\Omega ,\text{ }\overline{u}=0%
\text{ \ on }\partial \Omega  \label{80}
\end{equation}%
\begin{equation}
-\Delta _{q(x)}\overline{v}=\lambda \left\{ 
\begin{array}{ll}
1 & \text{ in \ }\Omega \backslash \overline{\Omega }_{\delta } \\ 
\overline{v}^{-\beta _{2}(x)} & \text{\ in \ }\Omega _{\delta }%
\end{array}%
\right. ,\text{ }\overline{v}>0\text{ in }\Omega ,\text{ }\overline{v}=0%
\text{ \ on }\partial \Omega .  \label{80*}
\end{equation}%
where $\lambda >1$ is a constant and 
\begin{equation*}
\Omega _{\delta }=\left\{ x\in \Omega :d\left( x,\partial \Omega \right)
<\delta \right\} .
\end{equation*}%
Combining the results in \cite[Lemmas 1 and 3]{AM} and \cite{Fan}, it is
readily seen that for $\lambda >1$ large $\overline{u}$ and $\overline{v}$
verify 
\begin{equation}
\begin{array}{l}
\min \{\delta ,d(x)\}\leq \overline{u}(x)\leq c_{1}\lambda ^{\frac{1}{p^{-}-1%
}}\text{ \ in }\Omega ,%
\end{array}%
\end{equation}%
and%
\begin{equation}
\begin{array}{l}
\min \{\delta ,d(x)\}\leq \overline{v}(x)\leq c_{2}\lambda ^{\frac{1}{q^{-}-1%
}}\text{ \ in }\Omega ,%
\end{array}
\label{2}
\end{equation}%
for some positive constant $c_{1},$ $c_{2}$ independent of $\lambda $ and
for $\delta >0$ small. Moreover, similar arguments explored in the proof of 
\cite[Theorem 4.4]{QZ} produce constants $c_{0},c_{0}^{\prime }>0$ such that%
\begin{equation}
\begin{array}{l}
\overline{u}(x)\leq c_{0}d(x)^{\theta _{1}}\text{ \ and \ }\overline{v}%
(x)\leq c_{0}^{\prime }d(x)^{\theta _{2}}\text{ in }\Omega _{\delta },%
\end{array}
\label{82}
\end{equation}%
for some constants $\theta _{1},\theta _{2}\in (0,1),$ which assumed to
satisfy the estimates%
\begin{equation}
\begin{array}{l}
\theta _{1}\geq \frac{-\beta _{1}^{-}}{\alpha _{1}^{-}}\text{ \ and \ }%
\theta _{2}\geq \frac{-\alpha _{2}^{-}}{\beta _{2}^{-}}.%
\end{array}%
\end{equation}%
Notice that $\theta _{1}$ and $\theta _{2}$ exist since $-\beta
_{1}^{-}<\alpha _{1}^{-}$ and $-\alpha _{2}^{-}<\beta _{2}^{-}$ (see \textrm{%
(H.}$f$\textrm{)} and \textrm{(H.}$g$\textrm{)}).

Now, let consider the functions $\underline{u}$ and $\underline{v}$ defined
by 
\begin{equation}
-\Delta _{p(x)}\underline{u}=\lambda ^{-1}\left\{ 
\begin{array}{ll}
1 & \text{ in \ }\Omega \backslash \overline{\Omega }_{\delta } \\ 
-1 & \text{\ in \ }\Omega _{\delta }%
\end{array}%
\right. ,\text{ }\underline{u}=0\text{ \ on }\partial \Omega  \label{83}
\end{equation}%
and 
\begin{equation}
-\Delta _{q(x)}\underline{v}=\lambda ^{-1}\left\{ 
\begin{array}{ll}
1 & \text{ in \ }\Omega \backslash \overline{\Omega }_{\delta } \\ 
-1 & \text{\ in \ }\Omega _{\delta }%
\end{array}%
\right. ,\text{ }\underline{v}=0\text{ \ on }\partial \Omega .  \label{83*}
\end{equation}%
where $\Omega _{\delta }$ is given by%
\begin{equation}
\Omega _{\delta }=\left\{ x\in \Omega :d\left( x,\partial \Omega \right)
<\delta \right\} ,  \label{70}
\end{equation}%
with a fixed $\delta >0$ sufficiently small. Combining \cite[Lemma 2.1]{Fan}
and \cite[Theorem 1.1]{FZZ} with \cite[Lemma 3]{AM}, we get 
\begin{equation}
\begin{array}{l}
c_{3}d(x)\leq \underline{u}(x)\leq c_{4}\lambda ^{\frac{-1}{p^{+}-1}}\text{
\ and \ }c_{3}^{\prime }d(x)\leq \underline{v}(x)\leq c_{4}^{\prime }\lambda
^{\frac{-1}{q^{+}-1}}\text{ \ in }\Omega ,%
\end{array}
\label{84}
\end{equation}%
where $c_{3},c_{4},c_{3}^{\prime }$ and $c_{4}^{\prime }$ are positive
constants. Obviously, from (\ref{80}), (\ref{80*}), (\ref{83}) and (\ref{83*}%
), we have $(\underline{u},\underline{v})\leq (\overline{u},\overline{v})$
in $\overline{\Omega }$ for $\lambda >0$ large.

The following result allows us to achieve useful comparison properties.

\begin{proposition}
\label{P1}Assume that \textrm{(H.}$f$\textrm{)}, \textrm{(H.}$g$\textrm{)}
and \textrm{H(}$f,g$\textrm{)}$_{2}$ hold. Then, for $\lambda >0$ large
enough, we have%
\begin{equation}
-\Delta _{p(x)}\underline{u}\leq f(\underline{u},\overline{v})\text{, \ \ }%
-\Delta _{q(x)}\underline{v}\leq g(\overline{u},\underline{v})\text{ \ in }%
\Omega ,  \label{21}
\end{equation}%
\begin{equation}
-\Delta _{p(x)}\overline{u}\geq f(\overline{u},\underline{v}),\text{ \ }%
-\Delta _{q(x)}\overline{v}\geq g(\underline{u},\overline{v})\text{ \ in }%
\Omega .  \label{21*}
\end{equation}
\end{proposition}

\begin{proof}
For all $\lambda >0$ one has%
\begin{equation}
-\lambda ^{-1}\underline{u}^{-(p^{-}-1)}\leq 0<1\text{ \ and \ }-\lambda
^{-1}\underline{v}^{-(q^{-}-1)}\leq 0<1\text{ \ in }\Omega _{\delta }.
\label{32}
\end{equation}%
By (\ref{84}), it follows that%
\begin{equation}
\lambda ^{-1}\underline{u}^{-(p^{-}-1)}\leq \lambda
^{-1}(c_{3}d(x))^{-(p^{-}-1)}\leq \lambda ^{-1}(c_{3}\delta
)^{-(p^{-}-1)}\leq 1\text{ \ in }\Omega \backslash \overline{\Omega }%
_{\delta },  \label{35}
\end{equation}%
and%
\begin{equation}
\lambda ^{-1}\underline{v}^{-(q^{-}-1)}\leq \lambda ^{-1}(c_{3}^{\prime
}d(x))^{-(q^{-}-1)}\leq \lambda ^{-1}(c_{3}^{\prime }\delta
)^{-(q^{-}-1)}\leq 1\text{ \ in }\Omega \backslash \overline{\Omega }%
_{\delta },  \label{35*}
\end{equation}%
provided that $\lambda $ is sufficiently large. Another hand, by \textrm{H(}$%
f,g$\textrm{)}$_{2}$\textrm{\ }there exist constants $\rho ,\bar{\rho}>0$
such that%
\begin{equation}
\begin{array}{l}
f(s_{1},s_{2})\geq s_{1}^{p^{-}-1}\text{, \ for all }0<s_{1}<\rho ,\text{\
for all }0<s_{2}<\lambda ^{\frac{1}{p^{-}-1}},%
\end{array}
\label{36}
\end{equation}%
and%
\begin{equation}
\begin{array}{l}
g(s_{1},s_{2})\geq s_{2}^{q^{-}-1}\text{, \ for all }0<s_{1}\leq \lambda ^{%
\frac{1}{q^{-}-1}},\text{\ for all }0<s_{2}<\bar{\rho}.%
\end{array}
\label{36*}
\end{equation}%
Then, for $\lambda >0$ sufficiently large so that 
\begin{equation*}
\max \{c_{4}\lambda ^{\frac{-1}{p^{-}-1}},c_{4}^{\prime }\lambda ^{\frac{-1}{%
q^{-}-1}}\}<\min \{\rho ,\bar{\rho}\},
\end{equation*}%
combining (\ref{32}) - (\ref{36*}) together, we infer that (\ref{21}) holds
true.

Next, we show (\ref{21*}). By \textrm{(H.}$f$\textrm{)}, \textrm{(H.}$g$%
\textrm{)}, (\ref{84}), (\ref{1}) and (\ref{2}), it follows that%
\begin{equation}
\begin{array}{l}
f(\overline{u},\underline{v})\leq M_{1}(1+\overline{u}^{\alpha _{1}(x)})(1+%
\underline{v}^{\beta _{1}(x)}) \\ 
\leq M_{1}(1+c_{4}^{\alpha _{1}(x)}\lambda ^{\frac{\alpha _{1}^{+}}{p^{-}-1}%
})(1+(c_{3}^{\prime }d(x))^{\beta _{1}(x)})\leq \lambda \text{ in \ }\Omega
\backslash \overline{\Omega }_{\delta }%
\end{array}
\label{5}
\end{equation}%
and%
\begin{equation}
\begin{array}{l}
g(\underline{u},\overline{v})\leq M_{2}(1+\underline{u}^{\alpha _{2}(x)})(1+%
\overline{v}^{\beta _{2}(x)}) \\ 
\leq M_{2}\left( 1+(c_{3}d(x))^{\alpha _{2}(x)}\right) (1+(c_{4}^{\prime
})^{\beta _{2}(x)}\lambda ^{\frac{\beta _{2}^{+}}{q^{-}-1}})\leq \lambda 
\text{ \ in }\Omega \backslash \overline{\Omega }_{\delta }\text{.}%
\end{array}
\label{5*}
\end{equation}%
provided that $\lambda >0$ is large enough. Now we deal with the
corresponding estimates on $\Omega _{\delta }$. From \textrm{(H.}$f$\textrm{)%
}, \textrm{(H.}$g$\textrm{)}, (\ref{84}), (\ref{1}), (\ref{2}) and (\ref{3}%
), we get%
\begin{equation}
\begin{array}{l}
\overline{u}^{\alpha _{1}(x)}f(\overline{u},\underline{v})\leq M_{1}(%
\overline{u}^{\alpha _{1}(x)}+\overline{u}^{2\alpha _{1}(x)})(1+\underline{v}%
^{\beta _{1}(x)}) \\ 
\leq M_{1}\left( (c_{0}d(x)^{\theta _{1}})^{\alpha
_{1}(x)}+(c_{0}d(x)^{\theta _{1}})^{2\alpha _{1}(x)}\right) \left(
1+(c_{3}^{\prime }d(x))^{\beta _{1}(x)}\right) \\ 
\leq M_{1}\max \{(c_{0})^{\alpha _{1}(x)},(c_{0})^{2\alpha
_{1}(x)}\}(d(x)^{\theta _{1}\alpha _{1}(x)}+d(x)^{2\theta _{1}\alpha
_{1}(x)})\left( 1+(c_{3}^{\prime }d(x))^{\beta _{1}(x)}\right) \\ 
\leq C_{1}(d(x)^{\theta _{1}\alpha _{1}^{-}}+d(x)^{2\theta _{1}\alpha
_{1}^{-}})\left( 1+d(x)^{\beta _{1}^{-}}\right) \leq \lambda \text{ in \ }%
\Omega _{\delta }%
\end{array}
\label{6}
\end{equation}%
and similarly%
\begin{equation}
\begin{array}{l}
\overline{v}^{\beta _{2}(x)}g(\underline{u},\overline{v})\leq M_{2}(1+%
\underline{u}^{\alpha _{2}(x)})(\overline{v}^{\beta _{2}(x)}+\overline{v}%
^{2\beta _{2}(x)}) \\ 
\leq M_{2}\left( 1+(c_{3}d(x))^{\alpha _{2}(x)}\right) ((c_{0}^{\prime
}d(x)^{\theta _{2}})^{\beta _{2}(x)}+(c_{0}^{\prime }d(x)^{\theta
_{2}})^{2\beta _{2}(x)})\leq \lambda \text{ \ in }\Omega _{\delta },%
\end{array}
\label{6*}
\end{equation}%
provided that $\lambda >0$ is sufficiently large. Consequently, (\ref{5}), (%
\ref{5*}), (\ref{6}) and (\ref{6*}) allow to infer that (\ref{21*}) holds.
This ends the proof.
\end{proof}

\section{Proof of Theorem \protect\ref{T1}}

\label{S4}

For every $z_{1},z_{2}\in C_{0}^{1}(\Omega ),$ let us state the auxiliary
problem%
\begin{equation}
\left\{ 
\begin{array}{ll}
-\Delta _{p(x)}u=\tilde{f}(z_{1},z_{2}) & \text{in }\Omega , \\ 
-\Delta _{q(x)}v=\tilde{g}(z_{1},z_{2}) & \text{in }\Omega , \\ 
u,v=0 & \text{on }\partial \Omega ,%
\end{array}%
\right.  \tag{$P_{z}$}  \label{pz}
\end{equation}%
where 
\begin{equation}
\begin{array}{c}
\tilde{f}(z_{1},z_{2})=f(\tilde{z}_{1},\tilde{z}_{2})\text{ \ and }\tilde{g}%
(z_{1},z_{2})=g(\tilde{z}_{1},\tilde{z}_{2}),%
\end{array}
\label{38}
\end{equation}%
with%
\begin{equation}
\tilde{z}_{i}=\min \left\{ \max \{z_{i},c_{0}d(x)\},\text{ }R\right\} \text{
for }i=1,2.  \label{39}
\end{equation}%
On account of (\ref{39}) it follows that $c_{0}d(x)\leq \tilde{z}_{i}\leq R$
for $i=1,2$.

\vspace{0.5cm}

The next result establishes an a priori estimate for system (\ref{pz}). In
addition, it shows that solutions $(u,v)$ of problem (\ref{pz}) cannot occur
outside the rectangle $[c_{0}d(x),L_{R}]\times \lbrack c_{0}d(x),L_{R}],$
with a constant $L_{R}>0$ defined below.

\begin{proposition}
\label{P2}Assume \textrm{(H.}$f$\textrm{)}, \textrm{(H.}$g$\textrm{) }and (%
\ref{h1}) hold. Then all solutions $(u,v)$ of (\ref{pz}) belong to $%
C^{1,\gamma }(\overline{\Omega })\times C^{1,\gamma }(\overline{\Omega })$
for some $\gamma \in (0,1)$ and there is a positive constante $L_{R}$,
depending on $R$, such that%
\begin{equation}
\left\Vert u\right\Vert _{C^{1,\gamma }(\overline{\Omega })},\left\Vert
v\right\Vert _{C^{1,\gamma }(\overline{\Omega })}<L_{R}.  \label{31}
\end{equation}%
Moreover, it holds%
\begin{equation}
\begin{array}{c}
u(x),v(x)\geq c_{0}d(x)\text{ \ in }\Omega .%
\end{array}
\label{17}
\end{equation}
\end{proposition}

\begin{proof}
First, we prove the boundedness for solutions of (\ref{pz}) in $L^{\infty
}(\Omega )\times L^{\infty }(\Omega )$. To this end, we adapt the argument
which proves \cite[Lemma 2]{AM}. For each $k\in \mathbb{N}$, set 
\begin{equation*}
U_{k,R}=\{x\in \Omega \,:\,u(x)>kR\}\text{ \ and \ }V_{k,R}=\{x\in \Omega
\,:\,v(x)>kR\},
\end{equation*}%
where the constant $R>0$ is given by Proposition \ref{P3}. Since $u,v\in
L^{1}(\Omega )$, we have 
\begin{equation}
|U_{k,R}|,|V_{k,R}|\rightarrow 0\quad \mbox{as}\quad k\rightarrow +\infty .
\label{E3}
\end{equation}%
Using $(u-kR)^{+}$ and $(v-kR)^{+}$ as a test function in (\ref{pz}), we get 
\begin{equation}
\left\{ 
\begin{array}{c}
\int_{U_{k,R}}|\nabla u|^{p}dx=\int_{U_{k,R}}f(\tilde{z}_{1},\tilde{z}%
_{2})(u-kR)^{+}dx \\ 
\int_{V_{k,R}}|\nabla v|^{q}dx=\int_{V_{k,R}}g(\tilde{z}_{1},\tilde{z}%
_{2})(v-kR)^{+}dx.%
\end{array}%
\right.  \label{79}
\end{equation}%
By \textrm{(H.}$f$\textrm{) }and (\ref{39}) observe that%
\begin{equation*}
\begin{array}{l}
\int_{\Omega }|f(\tilde{z}_{1},\tilde{z}_{2})|^{N}\text{ }dx\leq
C_{1}\int_{\Omega }(1+\tilde{z}_{1}^{N\alpha _{1}(x)})(1+\tilde{z}%
_{2}^{N\beta _{1}(x)})dx \\ 
\leq C_{1}(1+R^{N\beta ^{+}})\int_{\Omega }(1+(c_{0}d(x))^{N\alpha
_{1}(x)})dx\leq \hat{C}_{1}\int_{\Omega }(1+d(x)^{N\alpha _{1}^{-}})dx.%
\end{array}%
\end{equation*}%
Since $N\alpha _{1}^{-}>-1$ (see (\ref{c1})), \cite[Lemma in page 726]{LM}
garantees that%
\begin{equation*}
\int_{\Omega }d(x)^{N\alpha _{1}^{-}}dx<\infty .
\end{equation*}

Then, it follows that $f(\tilde{z}_{1},\tilde{z}_{2})\in L^{N}(\Omega )$ and
therefore%
\begin{equation}
\left\Vert f(\tilde{z}_{1},\tilde{z}_{2})\right\Vert
_{L^{N}(U_{k,R})}\rightarrow 0\text{ \ as }k\rightarrow +\infty .  \label{E1}
\end{equation}%
Similarly, we obtain%
\begin{equation}
\left\Vert g(\tilde{z}_{1},\tilde{z}_{2})\right\Vert
_{L^{N}(V_{k,R})}\rightarrow 0\text{ \ as }k\rightarrow +\infty .
\end{equation}%
Now, proceeding analogously to the proof of \cite[Lemma 2]{AM} provides a
constant $k_{0}\geq 1$ such that 
\begin{equation*}
|u(x)|\text{, }|v(x)|\leq k_{0}R\text{ \ a.e in }\Omega .
\end{equation*}%
Consider now functions $w_{1}$ and $w_{2}$ defined by%
\begin{equation}
\left\{ 
\begin{array}{ll}
-\Delta w_{1}=\tilde{f}(z_{1},z_{2}) & \text{in }\Omega \\ 
w_{1}=0 & \text{on }\partial \Omega%
\end{array}%
\right. \text{ \ and \ }\left\{ 
\begin{array}{ll}
-\Delta w_{2}=\tilde{g}(z_{1},z_{2}) & \text{in }\Omega \\ 
w_{2}=0 & \text{on }\partial \Omega .%
\end{array}%
\right.  \label{3*}
\end{equation}%
On account of (\ref{38}), \textrm{(H.}$f$\textrm{)}, \textrm{(H.}$g$\textrm{)%
},\textrm{\ }(\ref{39}), (\ref{h1}) and (\ref{c1}), one has%
\begin{equation}
\begin{array}{c}
\tilde{f}(z_{1},z_{2})\leq C_{2}d(x)^{\alpha _{1}^{-}}\text{ \ and \ }\tilde{%
g}(z_{1},z_{2})\leq C_{2}^{\prime }d(x)^{\beta _{2}^{-}}\text{ \ in }\Omega ,%
\end{array}
\label{42}
\end{equation}%
for some positive constants $C_{2}$ and $C_{2}^{\prime }$. On the basis of (%
\ref{c1}) and Thanks to \cite[Lemma in page 726]{LM}, the right-hand side of
problems in (\ref{3*}) belongs to $H^{-1}(\Omega )$. Consequently, the
Minty-Browder theorem (see \cite[Theorem V.15]{B}) implies the existence and
uniqueness of $w_{1}$ and $w_{2}$ in (\ref{3*}). Moreover, bearing in mind (%
\ref{c1}) and (\ref{42}), the regularity theory found in \cite[Lemma 3.1]%
{Hai} implies that $w_{1}$ and $w_{2}$ are bounded in $C^{1,\gamma }(%
\overline{\Omega })$, for certain $\gamma \in (0,1)$.

Thereby, subtracting (\ref{3*}) from (\ref{pz}) yields 
\begin{equation*}
-div(|\nabla u|^{p(x)-2}\nabla u-\nabla w_{1})=0\text{ \ and \ }-div(|\nabla
v|^{q(x)-2}\nabla v-\nabla w_{2})=0,
\end{equation*}%
and the $C^{1,\alpha }$-boundedness of $u$ and $v$ follows from \cite[%
Theorem 1.2]{Fan2}. Summarizing, we have obtained that solutions $(u,v)$ of (%
\ref{pz}) belong to $C^{1,\gamma }(\overline{\Omega })\times C^{1,\gamma }(%
\overline{\Omega })$, for certain $\gamma \in (0,1)$, and there exists a
constant $L_{R}>0$ such that (\ref{31}) holds. Furthermore, a quite similar
argument showing the second part of Proposition \ref{P3} leads to (\ref{17}%
). This completes the proof.
\end{proof}

\bigskip

Next we prove the existence result for cooperative system (\ref{p}).

\begin{proof}[Proof of Theorem \protect\ref{T1}]
Denote by%
\begin{equation*}
\begin{array}{c}
\mathcal{B}(0,L_{R})=\{(u,v)\in C^{1}(\overline{\Omega })\times C^{1}(%
\overline{\Omega }):\left\Vert u\right\Vert _{C^{1}(\overline{\Omega }%
)}+\left\Vert v\right\Vert _{C^{1}(\overline{\Omega })}<L_{R}\}%
\end{array}%
\end{equation*}%
and%
\begin{equation*}
\begin{array}{c}
\mathcal{O}=\{(u,v)\in \mathcal{B}(0,L_{R}):u(x),v(x)\geq c_{0}d(x)\text{ \
in }\Omega \}.%
\end{array}%
\end{equation*}%
Let us introduce the operator $\mathcal{P}:\mathcal{O}\rightarrow C(%
\overline{\Omega })\times C(\overline{\Omega })$ by $\mathcal{P}%
(z_{1},z_{2})=(u,v)$, where $(u,v)$ is the solution of problem (\ref{pz}).
Bearing in mind (\ref{42}) and (\ref{c1}), the Minty-Browder theorem
together with \cite[Lemma 2]{AM} garantee that problem (\ref{pz}) has a
unique solution $(u,v)$ in $C^{1,\gamma }(\overline{\Omega })\times
C^{1,\gamma }(\overline{\Omega })$, for certain $\gamma \in (0,1)$. This
ensures the operator $\mathcal{P}$ is well defined. Moreover, analysis
similar to that in the proof of Theorem $3$ in \cite{AM} imply that $%
\mathcal{P}$ is continuous and compact operator. On the other hand,
according to Proposition \ref{P2}, it follows that $\mathcal{O}$ is
invariant by $\mathcal{P}$, that is, $\mathcal{P}(\mathcal{O})\subset 
\mathcal{O}$. Therefore we are in a position to apply Schauder's fixed point
Theorem to the set $\mathcal{O}$ and the map $\mathcal{P}:\mathcal{O}%
\rightarrow \mathcal{O}$. This ensures the existence of $(u,v)\in \mathcal{O}
$ satisfying $\mathcal{P}(u,v)=(u,v)$, that is, $(u,v)\in C^{1}(\overline{%
\Omega })\times C^{1}(\overline{\Omega })$ is a solution of problem%
\begin{equation*}
\left\{ 
\begin{array}{ll}
-\Delta _{p(x)}u=\tilde{f}(u,v) & \text{in }\Omega , \\ 
-\Delta _{q(x)}v=\tilde{g}(u,v) & \text{in }\Omega , \\ 
u,v=0 & \text{on }\partial \Omega .%
\end{array}%
\right.
\end{equation*}%
Finally, thank's to proposition \ref{P3}, it turns out that $(u,v)\in C^{1}(%
\overline{\Omega })\times C^{1}(\overline{\Omega })$ is a (positive)
solution of problem (\ref{p}).
\end{proof}

\section{Proof of Theorem \protect\ref{T3}}

\label{S5}

The proof is based on Schauder's fixed point Theorem. Using the functions $(%
\underline{u},\underline{v})$ and $(\overline{u},\overline{v})$ given in (%
\ref{80}), (\ref{80*}), (\ref{83}) and (\ref{83*}) let introduce the set 
\begin{equation*}
\mathcal{K}=\left\{ (y_{1},y_{2})\in C(\overline{\Omega })\times C(\overline{%
\Omega }):\underline{u}\leq y_{1}\leq \overline{u}\text{ and }\underline{v}%
\leq y_{2}\leq \overline{v}\text{ in }\Omega \right\} ,
\end{equation*}%
which is closed, bounded and convex in $C(\overline{\Omega })\times C(%
\overline{\Omega })$. Then we define the operator $\mathcal{T}:\mathcal{K}%
\rightarrow C(\overline{\Omega })\times C(\overline{\Omega })$ by $\mathcal{T%
}(y_{1},y_{2})=(u,v)$, where $(u,v)$ is required to satisfy 
\begin{equation}
\left\{ 
\begin{array}{l}
-\Delta _{p(x)}u=f(y_{1},y_{2})\text{ in }\Omega \\ 
-\Delta _{q(x)}v=g(y_{1},y_{2})\text{ in }\Omega \\ 
u,v=0\text{ on }\partial \Omega .%
\end{array}%
\right.  \tag{$P_{y}$}  \label{pr}
\end{equation}%
For $(y_{1},y_{2})\in \mathcal{K}$, we derive from \textrm{(H.}$f$\textrm{)}%
, $\mathrm{(H.}g\mathrm{)}$, (\ref{1}), (\ref{2}), (\ref{82}), and (\ref{84}%
) the estimates 
\begin{equation}
\begin{array}{l}
f(y_{1},y_{2})\leq m_{1}(1+\overline{u}^{\alpha _{1}(x)})(1+\underline{v}%
^{\beta _{1}(x)})\leq C_{1}d(x)^{\beta _{1}(x)}\text{ \ in }\Omega%
\end{array}
\label{28}
\end{equation}%
and 
\begin{equation}
\begin{array}{l}
g(y_{1},y_{2})\leq m_{2}(1+\underline{u}^{\alpha _{2}(x)})(1+\overline{v}%
^{\beta _{2}(x)})\leq C_{2}d(x)^{\alpha _{2}(x)}\text{ \ in }\Omega ,%
\end{array}
\label{29}
\end{equation}%
with positive constants $C_{1}$, $C_{2}$. We point out that estimates (\ref%
{28}) and (\ref{29}) combined with (\ref{c2}) and (\ref{c2*}) enable us to
deduce that $f(y_{1},y_{2})\in W^{-1,p^{\prime }(x)}(\Omega )$ and $%
g(y_{1},y_{2})\in W^{-1,q^{\prime }(x)}(\Omega )$. Then the unique
solvability of $(u,v)$ in (\ref{pr}) is readily derived from Minty-Browder
theorem (see, e.g., \cite{B}). Hence, the operator $\mathcal{T}$ is well
defined.

Using the regularity theory up to the boundary (see \cite[Lemma 2]{AM}), it
follows that $(u,v)\in C^{1,\beta }(\overline{\Omega })\times C^{1,\beta }(%
\overline{\Omega })$, with some $\beta \in (0,1)$, and there is a constant $%
M>0$ such that $\Vert u\Vert _{C^{1,\beta }(\overline{\Omega })},\Vert
v\Vert _{C^{1,\beta }(\overline{\Omega })}\leq M,$ whenever $(u,v)=\mathcal{T%
}(y_{1},y_{2})$ with $(y_{1},y_{2})\in \mathcal{K}$. Then, analysis similar
to that in the proof of \cite[Theorem $3$]{AM} imply that $\mathcal{T}$ is
continuous and compact operator.

The next step in the proof is to show that $\mathcal{T}(\mathcal{K})\subset 
\mathcal{K}$. Let $(y_{1},y_{2})\in \mathcal{K}$ and denote $\left(
u,v\right) =\mathcal{T}(y_{1},y_{2})$. Using the definitions of $\mathcal{K}$
and $\mathcal{T}$, on the basis of Proposition \ref{P1}, $(\mathrm{H}.f)$
and $(\mathrm{H.}g)$, it follows that 
\begin{equation*}
\begin{array}{l}
-\Delta _{p(x)}u(x)=f(y_{1}(x),y_{2}(x))\leq f(\overline{u}(x),\underline{v}%
(x))\leq -\Delta _{p(x)}\overline{u}(x)\text{ in }\Omega ,%
\end{array}%
\end{equation*}%
and similarly 
\begin{equation*}
\begin{array}{l}
-\Delta _{q(x)}v(x)=g(y_{1}(x),y_{2}(x))\leq g(\underline{u}(x),\overline{v}%
(x))\leq -\Delta _{q(x)}\overline{v}(x)\text{ in }\Omega .%
\end{array}%
\end{equation*}%
Proceeding in the same way, via Proposition \ref{P1} and hypotheses $(%
\mathrm{H}.f)$, $(\mathrm{H.}g)$, leads to 
\begin{equation*}
\begin{array}{l}
-\Delta _{p(x)}u(x)=f(y_{1}(x),y_{2}(x))\geq f(\underline{u},\overline{v}%
)\geq -\Delta _{p(x)}\underline{u}(x)\text{ \ in }\Omega ,%
\end{array}%
\end{equation*}%
and similarly 
\begin{equation*}
\begin{array}{l}
-\Delta _{q(x)}v(x)=g(y_{1}(x),y_{2}(x))\geq g(\overline{u},\underline{v}%
)\geq -\Delta _{q(x)}\underline{v}(x)\text{ in }\Omega .%
\end{array}%
\end{equation*}%
Then from the strict monotonicity of the operators $-\Delta _{p(x)}$ and $%
-\Delta _{q(x)}$ we get that $\left( u,v\right) \in \mathcal{K}$, which
establishes that $\mathcal{T}(\mathcal{K})\subset \mathcal{K}$. Therefore we
are in a position to apply Schauder's fixed point Theorem to the set $%
\mathcal{K}$ and the map $\mathcal{T}:\mathcal{K}\rightarrow \mathcal{K}$.
This ensures the existence of $(u,v)\in \mathcal{K}$ satisfying $(u,v)=%
\mathcal{T}(u,v).$ Moreover, because the solution $(u,v)\in \mathcal{K}$ and 
$(\mathrm{H}.f),$ $(\mathrm{H.}g),$ (\ref{c2}) and (\ref{c2*}) are
fulfilled, we conclude from \cite[Lemma 2]{AM} that $(u,v)\in C^{1}(%
\overline{\Omega })\times C^{1}(\overline{\Omega })$. This ends the proof.

\end{document}